\documentclass[reqno]{amsart}
\usepackage{hyperref}

\begin{document}
\title[\hfilneg  Caristi Kirk's Theorem in $M-$Metric Type Spaces  \hfil ]
{ Relations Between partial Metric Spaces and M-Metric Spaces,
Caristi Kirk's Theorem in $M-$Metric Type  Spaces }
\author[K. Abodayeh, N. Mlaiki, T. Abdeljawad, W. Shatanawi \hfil  \hfilneg]{K. Abodayeh, N. Mlaiki, T. Abdeljawad, W. Shatanawi}

\address{K. Abodayeh, N. Mlaiki, T. Abdeljawad, W. Shatanawi,  \newline
Department of Mathematics and general courses\\
\noindent Prince Sultan University\\
\noindent Riyadh, Saudi Arabia.}
\email{\href{mailto:kamal@psu.edu.sa}{kamal@psu.edu.sa}}
\email{\href{mailto:nmlaiki@psu.edu.sa}{nmlaiki@psu.edu.sa}}
\email{\href{mailto:tabdeljawad@psu.edu.sa}{tabdeljawad@psu.edu.sa}}
\email{\href{mailto:wshatanawi@psu.edu.sa}{wshatanawi@psu.edu.sa}}

\address{W. Shatanawi,  \newline
Department of Mathematics, Hashemite University\\
Zarqa, Jordan.}
\email{\href{mailto:swasfi@hu.edu.jo}{swasfi@hu.edu.jo}}


\thispagestyle{empty}

\numberwithin{equation}{section}
\newtheorem{theorem}{Theorem}
\newtheorem{acknowledgement}[theorem]{Acknowledgement}
\newtheorem{algorithm}[theorem]{Algorithm}
\newtheorem{axiom}[theorem]{Axiom}
\newtheorem{case}[theorem]{Case}
\newtheorem{claim}[theorem]{Claim}
\newtheorem{conclusion}[theorem]{Conclusion}
\newtheorem{condition}[theorem]{Condition}
\newtheorem{conjecture}[theorem]{Conjecture}
\newtheorem{corollary}[theorem]{Corollary}
\newtheorem{criterion}[theorem]{Criterion}
\newtheorem{definition}[theorem]{Definition}
\newtheorem{example}[theorem]{Example}
\newtheorem{exercise}[theorem]{Exercise}
\newtheorem{lemma}[theorem]{Lemma}
\newtheorem{notation}[theorem]{Notation}
\newtheorem{problem}[theorem]{Problem}
\newtheorem{proposition}[theorem]{Proposition}
\newtheorem{remark}[theorem]{Remark}
\newtheorem{solution}[theorem]{Solution}
\newtheorem{summary}[theorem]{Summary}
\newcommand{\IR}{\mbox{I \hspace{-0.2cm}R}}
\newcommand{\IN}{\mbox{I \hspace{-0.2cm}N}}

\begin{abstract}
Very recently, Mehadi et al [M. Asadi, E. Karap{\i}nar, and P.
Salimi, New extension of partial metric spaces with some fixed-point
results on $M-$metric spaces] extended the partial metric spaces to
the notion of $M-$metric spaces. In this article, we study some
relations between partial metric spaces and $M-$metric spaces. Also,
we generalize Caristi Kirki's Theorem from partial metric spaces to
$M-$metric spaces, where we corrected some gaps in the proof of the
main Theorem in E. Karap\i nar [E. Karap\i nar, Generalizations of
Caristi Kirk's Theorem on Partial Metric Spaces, Fixed Point Theory
Appl. 2011: 4, (2011)].
 We close our contribution by introducing some examples to validate and verify our extension results.
\end{abstract}

\maketitle

\section{Introduction and Preliminaries}

Fixed point theory has many applications in applied sciences. So the
attraction of large number of scientists in this subject is
understood. The Banach contraction theorem \cite{BN} is the first
result in fixed point theorems. Many authors extended the notion of
standard metric spaces in many ways. In 1994 Mattews \cite{Mat92}
introduced the notion of partial metric spaces as a generalization
of standard metric spaces in the sense that the distance between
identical points need not be zero. Then after, many authors
formulated and proved many fixed and common fixed point theorems in
partial metric spaces \cite{EK00}-\cite{Shat1}. In 2013, Haghi et al
\cite{H} introduced an outstanding technique to reduce some fixed
and common fixed point theorems from partial metric spaces to
standard metric spaces. In 2014, M. Asadi et al \cite{New M}
formulated the concept of $M-$metric spaces as a generalization of
partial metric spaces, as well as, a generalization of metric
spaces. It is worth mention that the technique of Haghi et al
\cite{H} is not applicable in $M-$metric spaces.

Here, we present the basic definitions and crucial results of
partial metric spaces.

A partial metric is a function $p:X\times X\to [0,\infty)$ satisfying the following conditions
\begin{itemize}
\item[$(P1)$] If $p(x,x)=p(x,y)=p(y,y)$, then $x=y$,
\item[$(P2)$] $p(x,y)=p(y,x)$,
\item[$(P3)$] $p(x,x)\leq p(x,y)$,
\item[$(P4)$] $p(x,z)+p(y,y)\leq p(x,y)+p(y,z)$,
\end{itemize}
for all $x,y,z \in X$. Then $(X,p)$ is called a partial metric space.
 If $p$ is a partial metric $p$ on $X$,  then the function $d_p : X
\times X \rightarrow [0,\infty)$ given by
$$d_p(x,y)=2p(x,y)-p(x,x)-p(y,y)$$
is a metric on $X$. Each partial metric $p$ on
$X$ generates a $T_0$ topology $\tau_p$ on $X$ with a base of the
family of open $p$-balls $\{B_p(x, \varepsilon): x\in X,~ \varepsilon>
0\}$, where
$B_p(x, \varepsilon)=\{y \in X: p(x, y)< p(x, x)+\varepsilon\}$
for all $x \in X$ and $\varepsilon>0$.
Similarly, closed $p$-ball is defined as
$B_p[x, \varepsilon]=\{y \in X: p(x, y)\leq p(x, x)+\varepsilon\}$. For more details see e.g. \cite{IAFTPA2011,Mat92}.

\begin{definition}[See e.g. \cite{Mat92,IAFTPA2011}]
Let $(X,p)$ be a partial metric space.
\begin{enumerate}
\item[$(i)$] A sequence $\{x_n\}$ in $X$ converges to $x\in X$
whenever $\displaystyle\lim_{n\to \infty}p(x,x_n)=p(x,x)$,
\item[$(ii)$] A sequence $\{x_n\}$ in $X$ is called Cauchy whenever $\displaystyle\lim_{n,m\to \infty} p(x_n, x_m)$ exists (and finite),
\item[$(iii)$] $(X,p)$ is said to be complete if every Cauchy sequence
$\{x_n\}$ in $X$ converges, with respect to $\tau_p$, to a point
$x\in X$, that is, $\displaystyle\lim_{n,m_\to \infty} p(x_n, x_m)=p(x, x)$.
\item[$(iv)$] A mapping $f:X\to X$ is said to be continuous at $x_0 \in X$ if for each $\varepsilon> 0$ there
exists $\delta>0$ such that $f(B(x_0,\delta))\subset
B(f(x_0),\varepsilon)$.
\end{enumerate}
\end{definition}

\begin{lemma} \label{before}
[See e.g. \cite{Mat92,IAFTPA2011}]
Let $(X,p)$ be a partial metric space.
\begin{enumerate}
\item[$(a)$] A sequence $\{x_n\}$ is Cauchy if and only if
$\{x_n\}$ is a Cauchy sequence in the metric space $(X,d_p)$,
\item[$(b)$] $(X,p)$ is complete if and only if the metric space $(X,d_p)$ is complete. Moreover,
\begin{equation}
\displaystyle\lim_{n\to \infty}d_p(x,x_n)=0 \Leftrightarrow
\displaystyle\lim_{n\to \infty}p(x,x_n)=\displaystyle\lim_{n,m_\to \infty} p(x_n, x_m)=p(x, x).
\label{elma1}
\end{equation}
\end{enumerate}
\label{eks}
\end{lemma}
\begin{lemma} (See e.g. \cite{EK00,EK02})
Let $(X,p)$ be a  partial metric space. Then
\begin{enumerate}
\item[(A)] If $p(x,y)=0$ then $x=y$.
\item[(B)] If $x\neq y$, then  $p(x,y)>0$.
\end{enumerate}
\label{tlemma}
\end{lemma}
\begin{remark}
If $x=y $, $p(x,y)$ may not be $0$.
\end{remark}
The following two lemmas can be derived from the triangle inequality $(P4)$.
\begin{lemma}  (See e.g. \cite{EK00,EK02})
Let $x_n\rightarrow z$ as $n\rightarrow\infty$ in a partial metric space $(X,p)$ where $p(z,z)=0$.
Then $\displaystyle\lim_{n\rightarrow\infty} p(x_n,y)=p(z,y)$ for every $y \in X$.
\label{sonra}
\end{lemma}
\begin{lemma}  (See e.g. \cite{SES})
Let  $\lim_{n \rightarrow \infty} p(x_n,y)=p(y,y)$ and  $\lim_{n \rightarrow \infty} p(x_n,z)=p(z,z)$.
If $p(y,y)=p(z,z)$ then $y=z.$
\label{ses}
\end{lemma}
\begin{remark}
Limit of a sequence $\{x_n\}$ in a partial metric space $(X,p)$ is not unique.
\end{remark}
\begin{example}
Consider $X=[0,\infty)$ with $p(x,y)=\max \{x,y\}$. Then
$(X,p)$ is a  partial metric space. Clearly, $p$ is not a metric.
Observe that  the sequence $\{1+\frac{1}{n^2}\}$ converges both for example  to $x=2$  and  $y=3$, so no uniqueness of the limit.
\end{example}

Here, we present the definition of $M-$metric spaces.

\noindent{\bf Notation.}
Let $m:X\times X\to [0,\infty)$ be a function. The following notations are useful for our investigation.
\begin{enumerate}
\item $m_{x,y}=\min \{ m(x,x),m(y,y) \} .$
\item $M_{x,y}=\max \{ m(x,x),m(y,y) \} .$
\end{enumerate}

\begin{definition} \cite{New M}
Let $X$ be a nonempty set. A function $m : X \times X \rightarrow [0,\infty)$ is called an {\it M-metric}
if the following conditions are satisfied:
\begin{enumerate}
\item $m(x, x) = m(y, y) = m(x, y)\Leftrightarrow x = y,$
\item $m_{x,y} \leq m(x, y),$
\item $m(x, y) = m(y, x),$
\item $(m(x, y) - m_{x,y}) \leq (m(x, z) - m_{x,z})+(m(z, y) -m_{z,y}).$
\end{enumerate}
Then the pair $(X, m)$ is called an M-metric space.
\end{definition}
It is straightforward to verify that every partial metric space is
an $M-$metric space but the converse is not true.
\begin{example}
Let $X=\{1,2,3 \}$. Define $m$ on $X\times X$ as follows: \\
$m(1,1)=1$, $m(2,2)=4$, $m(3,3)=5$, $m(1,2)=m(2,1)=10$, $m(1,3)=m(3,1)=7, m(2,3)=m(3,2)=6$. \\
It is easy to verify that $(X,m)$ is an M-metric space but it is not a partial metric space; to
see this, $m(1,2)\not\leq m(1,3)+m(3,2)-m(3,3)$.
\end{example}

For more about convergence and completeness in $M-$metric spaces, we refer to \cite{New M}.

\section{Convergence in Partial Metric Spaces and M-Metric Spaces}

In this section, we give some new concepts about partial metric
space topologies and the relation to the topology of the $M$-metric
space studied in \cite{New M}.
\begin{definition} (symmetric convergence)
Let $\{x_n\}$ be a sequence in a partial metric space $(X,p)$ with $B_p(x,\epsilon)=\{y \in X: p(x,y)< \epsilon+p(x,x)\}$. We say that the sequence
$\{x_n\}$ converges symmetrically to $x \in X$  (shortly $x_n \rightarrow ^s x$), if for every $\epsilon >0$ there exists $n_0 \in \mathbb{N}$ such that
\begin{equation}\label{con}
  x_n \in B_p(x,\epsilon)~~\texttt{and}~~x \in B_p(x_n,\epsilon)~~\texttt{for all}~~ n\geq n_0.
\end{equation}
Equivalently, if $\lim_{n\rightarrow \infty} p(x_n,x)=lim_{n\rightarrow \infty} p(x_n,x_n)=p(x,x)$.
\end{definition}
Clearly, symmetric convergence in partial metric spaces implies convergence and both coincide in metric spaces and in $M-$metric spaces. Symmetric convergence in partial metric spaces implies being Cauchy and the symmetric limit is unique. Namely, we have

\begin{lemma} \label{compllemma} (see Lemma 1 in  \cite{compl})
Let $(X,p)$ be a partial metric space and assume $\{x_n\}$ in $X$.
\begin{itemize}
  \item a) If $\{x_n\}$ is symmetrically convergent in $X$, then it is Cauchy.
  \item b)If $\{x_n\}$ converges symmetrically in $X$ to both $x$ and $y$, then $x=y$.
  \item c) If $\{x_n\}$ and $\{y_n\}$ converge symmetrically to $x$ and $y$,respectively. Then, $\lim_{n\rightarrow\infty}p(x_n,y_n)=p(x,y)$ .
\end{itemize}
\end{lemma}

\textbf{Notation}: For the convergence in a partial metric space (with respect to $\tau_p$) we denote $x_n\rightarrow^u x$ or $x_n\rightarrow x$ and for symmetric convergence we denote $x_n\rightarrow^s x$.

\indent
\begin{definition}
Let $f:(X,p)\rightarrow (Y,\rho)$ be a function between two partial metric spaces. Then,

\begin{itemize}
  \item a) $f$ is said to be $uu-$continuous at $a \in X$, if for every open ball $B_\rho(f(a),\epsilon)$ of $f(a)$ in $Y$, there exists an open ball $B_p(a,\delta)$ of $a$ in $X$ such that $f(B_p(a,\delta))\subset B_\rho(f(a),\epsilon)$. $f$ is called $uu-$continuous on $X$ if it is $uu-$continuous at each $a\in X$.
  \item b)$f$ is said to be $su-$continuous at $a \in X$, if for every open ball $B_\rho(f(a),\epsilon)$ of $f(a)$ in $Y$, there exists an open ball $B_p(a,\delta)$ of $a$ in $X$ such that $f(x) \in B_\rho(f(a),\epsilon)$ for any $x \in B_p(a,\delta) $ with $a \in B(x,\delta)$. $f$ is called $su-$continuous on $X$ if it is $su-$continuous at each $a\in X$.
  \item c)$f$ is said to be $us-$continuous at $a \in X$, if for every open ball $B_\rho(f(a),\epsilon)$ of $f(a)$ in $Y$, there exists an open ball $B_p(a,\delta)$ of $a$ in $X$ such that $f(x) \in B_\rho(f(a),\epsilon)$ and $f(a) \in B_\rho(f(x),\epsilon)$ for any $x \in B_p(a,\delta) $ . $f$ is called $us-$continuous on $X$ if it is $us-$continuous at each $a\in X$.
  \item d) )$f$ is said to be $ss-$continuous at $a \in X$, if for every open ball $B(f(a),\epsilon)$ of $f(a)$ in $Y$, there exists an open ball $B_p(a,\delta)$ of $a$ in $X$ such that $f(x) \in B_\rho(f(a),\epsilon)$ and $f(a) \in B_\rho(f(x),\epsilon)$ for any $x \in B_p(a,\delta) $ with $x \in B_p(a,\delta) $ . $f$ is called $ss-$continuous on $X$ if it is $ss-$continuous at each $a\in X$.
\end{itemize}

\end{definition}
The following theorem characterizes the above four continuity types by means of sequential ones via usual convergence and symmetric convergence. The proof is direct and classical.
\begin{theorem}
Let $f:(X,p)\rightarrow (Y,\rho)$ be a function between two partial metric spaces and $a \in X$. Then,

\begin{itemize}
  \item a) $f$ is $uu-$continuous at $a$ if and only if $f(x_n)\rightarrow^u f(a)$ for any sequence $x_n \in X$ with $x_n\rightarrow ^u a$.
  \item b) $f$ is $su-$continuous at $a$ if and only if $f(x_n)\rightarrow^s f(a)$ for any sequence $x_n \in X$ with $x_n\rightarrow ^u a$.
  \item c)  $f$ is $us-$continuous at $a$ if and only if $f(x_n)\rightarrow^u f(a)$ for any sequence $x_n \in X$ with $x_n\rightarrow ^s a$.
  \item d)$f$ is $ss-$continuous at $a$ if and only if $f(x_n)\rightarrow^s f(a)$ for any sequence $x_n \in X$ with $x_n\rightarrow ^s a$.
\end{itemize}

\end{theorem}

\begin{remark}
Since symmetric convergence implies usual convergence ($\tau_p-convergence$ ) in partial metric spaces, then it is easy to see the following implications:

\begin{itemize}
  \item a) $us-$implies $ss-$continuity.
  \item b) $ss-$implies $su-$continuity.
  \item c) $us-$implies $su-$continuity.
  \item d) $us-$implies $uu-$continuity.
  \item e) $uu-$implies $su-$continuity.
\end{itemize}
Also, since symmetric and $\tau_p-$convergence (usual convergence) coincide in metric spaces, then for example if $(Y,\rho)$ is metric space then $ss-$continuity and $su-$continuity are the same and $us-$ccontinuity and $uu-$continuity are the same. In this case $uu-$continuity will imply $ss-$continuity. For example, this is the case when the space $(Y,\rho)$ is the set of all nonnegative real  numbers $\mathbb{R}^+=[0,\infty)$ with absolute value metric topology. In this case we will have only two types of continuity the symmetric continuity and usual continuity. It is clear that symmetric continuity is weaker than continuity. For that reason, we next define symmetric (or weak)lower semi- continuous and lower semi-continuous real valued functions  on partial metric spaces. The symmetric (or weak) lower semi-continuity will be used generalize Caristi-Kirk Theorem in partial metric spaces.
\end{remark}

\begin{definition}
Assume $\phi:(X,p)\rightarrow [0,\infty)$ is a mapping of a partial metric space. Then, we say that $\phi(x)$ is lower semi-continuous  (symmetric (or weak) lower semi-continuous) at $x \in X$ if $\phi(x)\leq \liminf_n \phi(x_n)$ for any $x_n \in X$ with $x_n\rightarrow  x $ ($x_n\rightarrow ^s x$).
\end{definition}
Clearly, as mentioned above, that if $\phi(x)$ is lower semi-continuous then it is symmetric lower semi-continuous.

\begin{remark}
\begin{itemize}
  \item In \cite{New M}, it was mentioned that every partial metric space is $M-$ metric space but the partial metric topology used is not the same as the $M-$ metric topology. Hence, it is worth to notice that the $M-$ metric topology is the topology of symmetric convergence on the partial metric space $(X,p)$. Namely, $B_M(x,\epsilon)=\{y \in X: m(x,y)< m_{x,y}+\epsilon\}=\{y \in X: m(x,y)< m(x,x)+\epsilon ~\texttt{and}~m(x,y)< m(y,y)+\epsilon ~\}= \{y \in X: y \in B_p(x,\epsilon) ~\texttt{and}~ x \in B_p(y,\epsilon)~\}$.
  \item In Theorem 2.1 of \cite{New M}, the authors claimed that the topology of the $M-$metric space,$\tau_m$, is not Hausdorff. This is not accurate, since every metric space is $M-$mertic space (the proof has the gap that it is not in general possible to find the claimed $z$). It was to said that $M-$metric spaces are not necessary Hausdorff and a counter example was to be given. Also, the limit may not be unique (see Lemma 2.4 in \cite{New M}).
\end{itemize}

\end{remark}
\section{Generalizations to the Caristi Kirk's Theorem}
We first repair the gap in the proof of Theorem 5 in \cite{EK02}. Then, we remark a generalization by means of symmetric lower-semicontinuity.

\begin{theorem} (See Theorem 5 in \cite{EK02}) \label{A}
Let $(X,p)$ be a complete partial metric space and $\phi:X\rightarrow \mathbb{R}^+$ be lower semi-continuous. Assume $T:X\rightarrow X$ is a self mapping of $X$ satisfying the condition:

\begin{equation}\label{B}
  p(x,Tx)\leq \phi(x)-\phi(Tx),~~~~ \texttt{for all}~ x \in X.
  \end{equation}
  Assume also,
  \begin{equation}\label{Bb}
    X_0=\{x \in X: p(x,x)=0\}\neq \emptyset.
  \end{equation}

Then, $T$ has a fixed point.
\end{theorem}

\begin{proof}
For a fixed $x \in X$, define
\begin{equation}\label{B1}
  S(x)=\{z \in X: p(x,z)\leq p(x,x)+\phi(x)-\phi(z)\}~~\texttt{and}~~\alpha(x)=\inf\{\phi(z):z \in S(x)\}.
\end{equation}
Clearly, $x \in S(x)$ and hence $S(x)\neq \emptyset$. Also, it is clear that $0\leq \alpha (x)\leq \phi(x)$.
We take $x \in X$ and construct a sequence $\{x_n\}$ as follows:
$$x_1=x,~ x_{n+1} \in S(x_n)~\texttt{such that}~ \phi(x_{n+1})\leq \alpha(x_n)+\frac{1}{n}, ~\texttt{for}~n=2,3,...$$
Then, one can easily observe that for each $n \in \mathbb{N}$ we have
\begin{equation}\label{B2}
  0\leq p(x_n,x_{n+1}) -p(x_n,x_n)\leq \phi(x_n)-\phi(x_{n+1})
\end{equation}
and

\begin{equation}\label{B3}
  \alpha(x_n)\leq \phi(x_{n+1})\leq  \alpha(x_n) +\frac{1}{n}
\end{equation}
Notice that  the sequence $\{\phi(x_n)\}$ is a decreasing sequence of real numbers which is bounded below by zero. Hence, it will converge to a positive real number, say $L$, By means of (\ref{B3}) we see that
\begin{equation}\label{B4}
  L=\inf_{n}\phi(x_n)=\lim_{n\rightarrow \infty}\phi(x_n)=\lim \alpha(x_n).
\end{equation}
From (\ref{B3}) and (\ref{B4}), for each $k \in \mathbb{N}$, there exists $N_k\in \mathbb{N}$ such that
\begin{equation}\label{B5}
  \phi(x_n)\leq L+ \frac{1}{k},~\texttt{for all}~ n\geq N_k.
\end{equation}
By monotonicity of $\phi(x_n)$, for all $m\geq n\geq N_k$, we have
\begin{equation}\label{B6}
  L \leq \phi(x_m)\leq \phi(x_n)\leq L+ \frac{1}{k}.
\end{equation}
Hence, by noticing that $L-\phi(x_n) <0$, we reach at
\begin{equation}\label{B7}
  \phi(x_n)-\phi(x_m)<\frac{1}{k},~~\texttt{for all}~~m\geq n\geq N_k
\end{equation}
On the other hand, the triangle inequality together with (\ref{B2}) imply that

\begin{eqnarray} \ref{B8}
p(x_n,x_{n+2}) &\leq&  p(x_n,x_{n+1}) + p(x_{n+1},x_{n+2})- p(x_{n+1},x_{n+1}) \\  \nonumber
   &\leq& \phi(x_n)-\phi(x_{n+1})+p(x_n,x_n)+ \phi(x_{n+1}) \\   \nonumber
   &-&\phi(x_{n+2})+p(x_{n+1},x_{n+1})-p(x_{n+1},x_{n+1})\\ \nonumber
   &=& \phi(x_n)-\phi(x_{n+2})+p(x_n,x_n) \\
\end{eqnarray}
If we proceed inductively, we obtain that

\begin{equation}\label{B9}
   p(x_n,x_{m})\leq \phi(x_n)-\phi(x_{m})+p(x_n,x_n),~~\texttt{for all} m\geq n.
\end{equation}
\end{proof}
By (\ref{B7}),then it follows that
\begin{equation}\label{B10}
  p(x_n,x_{m})\leq \phi(x_n)-\phi(x_{m})+p(x_n,x_n)< \frac{1}{k}+ p(x_n,x_n),~~\texttt{for all} m \geq n \geq N_k.
\end{equation}
That is
\begin{equation}\label{B11}
 0\leq p(x_n,x_{m})- p(x_n,x_n)<\frac{1}{k},~~\texttt{for all} m \geq n \geq N_k.
\end{equation}
Letting $k$ or $m,n$ to tend to infinity in (\ref{B11}) we conclude that $\lim_{m,n\rightarrow \infty} [ p(x_n,x_{m})- p(x_n,x_n)]=0$.
Hence, $\lim_{m,n\rightarrow \infty}d_p(x_n,x_m)=\lim_{m,n\rightarrow \infty}[p(x_n,x_{m})-p(x_n,x_n)+p(x_n,x_{m})-p(x_m,x_m)]=0$. Then, completeness of the space $(X,p)$ implies by means of Lemma \ref{before} (b) that $(X,d_p)$ is complete and there exists $z \in X$ such that
\begin{equation}\label{B12}
  \lim_{n\rightarrow \infty}p(x_n,x)=\lim_{n\rightarrow \infty}p(x_n,x_n)=p(z,z).
\end{equation}
That is the sequence $\{x_n\}$ is symmetrically convergent to $z$. Then, by lower semi-continuity of $\phi(x)$,(\ref{B10}) and Lemma \ref{compllemma} (c), we have
\begin{eqnarray} \label{13}
  \phi(z) &\leq&  \liminf_m \phi(x_m)\\  \nonumber
   &\leq& \liminf_m[\phi(x_m)-p(x_n,x_m)+ p(x_n,x_n)] \\  \nonumber
  &=& \phi(x_n)-p(x_n,z)+ p(x_n,x_n)
\end{eqnarray}
or
\begin{equation}\label{B14}
  p(x_n,z)\leq \phi(x_n)-\phi(z)+ p(x_n,x_n)
\end{equation}
That is $z \in S(x_n)$ for all $n \in \mathbb{N}$ and thus $\alpha(x_n)\leq \phi(z)$. Taking (\ref{B4}) into account, we obtain $L\leq \phi (z)$. Moreover, by lower semi-continuity of $\phi(x)$ we know that $ \phi (z)\leq L$. Hence,  $ \phi (z)= L$.
Now, by the triangle inequality, the assumption (\ref{B}) which implies that $Tz \in S(z)$, and that $z \in S(x_n)$ for all $n \in \mathbb{N}$, we obtain

\begin{eqnarray}\label{B14}
  p(x_n,Tz) &\leq& p(x_n,z)+p(z,Tz)-p(z,z) \\ \nonumber
   &\leq& p(x_n,z)+p(z,Tz) \\
   &\leq& \phi(x_n)- \phi(z)+p(x_n,x_n)+\phi(z) -\phi(Tz)\\ \nonumber
   &=& \phi(x_n) -\phi(Tz)+p(x_n,x_n). \nonumber
\end{eqnarray}
Hence, $Tz \in S(x_n)$ for all $n \in \mathbb{N}$, which yields that $\alpha(x_n)\leq \phi(Tz)$ for all $n \in \mathbb{N}$
From (\ref{B4}) we have $\phi(Tz)\geq L$ is obtained and (\ref{B}) with $x=z$ implies  $\phi(Tz)\leq \phi (z)$. All together with that $\phi(z)=L$, we come to the conclusion that $\phi(Tz)\leq \phi (z)$ and hence from the assumption (\ref{B}),we have $p(Tz,z)=0$ and hence $Tz=z$.

\begin{remark}
\begin{itemize}
\item From the proof of Theorem \ref{A}, we noticed that lower semi-continuity of the function $\phi(x)$ can be replaced by weak (symmetric) lower semi-continuity. Indeed, in the proof, it has been shown that the sequence $\{x_n\}$ converges to $z$ symmetrically.
\item From the assumption (\ref{B}), it is clear that the fixed point $z$ in the proof Theorem \ref{A} must satisfy $p(z,z)=0$.
\item From the proof of Theorem \ref{A}, we noticed that the assumption (\ref{B}) can be replaced by a weaker one:
   \begin{equation}\label{wB}
                   p(x,Tx)\leq p(x,x)+\phi(x)-\phi(Tx),~~~~ \texttt{for all}~ x \in X.
   \end{equation}
                For example, the proof step (\ref{B14}) becomes

                \begin{eqnarray}
                 p(x_n,Tz) &\leq& p(x_n,z)+p(z,Tz)-p(z,z)\\ \nonumber
                   &\leq& \phi(x_n)- \phi(z)+p(x_n,x_n)+\phi(z) -\phi(Tz)+p(z,z)-p(z,z) \\ \nonumber
                   &=& \phi(x_n) -\phi(Tz)+p(x_n,x_n).
                \end{eqnarray}

              In this case we don't need the assumption (\ref{Bb}). Hence, it will be  more logical to consider the assumption type (\ref{wB}) as we shall see in the  next theorem
                when we generalize to $M-$metric spaces.
\end{itemize}

\end{remark}

Now, we generalize Theorem \ref{A} to $M-$metric spaces.

\begin{theorem}  \label{AA}
Let $(X,m)$ be a complete $M-$ metric space and $\phi:X\rightarrow \mathbb{R}^+$ be lower semi-continuous. Assume $T:X\rightarrow X$ is a self mapping of $X$ satisfying the condition:

\begin{equation}\label{BB}
  m(x,Tx)\leq m_{x,Tx}+ \phi(x)-\phi(Tx),~~~~ \texttt{for all}~ x \in X.
\end{equation}
Then, $T$ has a fixed point.
\end{theorem}

\begin{proof}
For a fixed $x \in X$, define
\begin{equation}\label{BB1}
  S(x)=\{z \in X: m(x,z)\leq m_{x,z}+\phi(x)-\phi(z)\}~~\texttt{and}~~\alpha(x)=\inf\{\phi(z):z \in S(x)\}.
\end{equation}
Clearly, $x \in S(x)$ and hence $S(x)\neq \emptyset$. Also, it is clear that $0\leq \alpha (x)\leq \phi(x)$.
We take $x \in X$ and construct a sequence $\{x_n\}$ as follows:
$$x_1=x,~ x_{n+1} \in S(x_n)~\texttt{such that}~ \phi(x_{n+1})\leq \alpha(x_n)+\frac{1}{n}, ~\texttt{for}~n=2,3,...$$
Then, one can easily observe that for each $n \in \mathbb{N}$ we have
\begin{equation}\label{BB2}
  0\leq m(x_n,x_{n+1}) -m_{x_n,x_{n+1}}\leq \phi(x_n)-\phi(x_{n+1})
\end{equation}
and

\begin{equation}\label{BB3}
  \alpha(x_n)\leq \phi(x_{n+1})\leq  \alpha(x_n) +\frac{1}{n}
\end{equation}
Notice that that the sequence $\{\phi(x_n)\}$ is a decreasing sequence of real numbers which is bounded below by zero. Hence, it will converge to a positive real number, say $L$, By means of (\ref{BB3}) we see that
\begin{equation}\label{BB4}
  L=\inf_{n}\phi(x_n)=\lim_{n\rightarrow \infty}\phi(x_n)=\lim \alpha(x_n).
\end{equation}
From (\ref{BB3}) and (\ref{BB4}), for each $k \in \mathbb{N}$, there exists $N_k\in \mathbb{N}$ such that
\begin{equation}\label{BB5}
  \phi(x_n)\leq L+ \frac{1}{k},~\texttt{for all}~ n\geq N_k.
\end{equation}
By monotonicity of $\phi(x_n)$, for all $m\geq n\geq N_k$,we have
\begin{equation}\label{BB6}
  L \leq \phi(x_m)\leq \phi(x_n)\leq L+ \frac{1}{k}.
\end{equation}
Hence, by noticing that $L-\phi(x_n) <0$, we reach at
\begin{equation}\label{BB7}
  \phi(x_n)-\phi(x_m)<\frac{1}{k},~~\texttt{for all}~~m\geq n\geq N_k
\end{equation}
On the other hand, the triangle inequality together with (\ref{BB2}) imply that

\begin{eqnarray} \ref{BB8}\nonumber
m(x_n,x_{n+2})-m_{x_n,x_{n+2}} &\leq&  m(x_n,x_{n+1})-m_{x_n,x_{n+1}} + m(x_{n+1},x_{n+2})- m_{x_{n+1},x_{n+2}} \\  \nonumber
   &\leq& \phi(x_n)-\phi(x_{n+1})+ \phi(x_{n+1})-\phi(x_{n+2}) \\   \nonumber
   &=& \phi(x_n)-\phi(x_{n+2}) \\
\end{eqnarray}
If we proceed inductively, we obtain that

\begin{equation}\label{BB9}
   m(x_n,x_{m})-m_{x_n,x_{m}}\leq \phi(x_n)-\phi(x_{m}),~~\texttt{for all} m\geq n.
\end{equation}

By (\ref{BB7}),then it follows that
\begin{equation}\label{BB10}
  m(x_n,x_{m})\leq \phi(x_n)-\phi(x_{m})+m_{x_n,x_m}< \frac{1}{k}+m_{x_n,x_m},~~\texttt{for all} m \geq n \geq N_k.
\end{equation}
That is
\begin{equation}\label{BB11}
 0\leq m(x_n,x_{m})- m_{x_n,x_m}<\frac{1}{k},~~\texttt{for all} m \geq n \geq N_k.
\end{equation}
Letting $k$ or $m,n$ to tend to infinity in (\ref{BB11}) we conclude that $\lim_{m,n\rightarrow \infty} [ m(x_n,x_{m})- m_{x_n,x_m}]=0$ and thus $\{x_n\}$ is Cauchy. Since $(X,m)$ is complete, there exists $z \in X$ such that
\begin{equation}\label{BB12}
  \lim_{n\rightarrow \infty}[m(x_n,x)-m_{x_n,x}]=0.
\end{equation}
 Then, by lower semi-continuity of $\phi(x)$,(\ref{BB10}) and Lemma 2.2 in  \cite{New M}, we have
\begin{eqnarray} \label{13}
  \phi(z) &\leq&  \liminf_m \phi(x_m)\\  \nonumber
   &\leq& \liminf_m[\phi(x_m)-m(x_n,x_m)+ m_{x_n,x_m}] \\  \nonumber
  &=& \phi(x_n)-m(x_n,z)+ m_{x_n,z}
\end{eqnarray}
or
\begin{equation}\label{BB14}
  m(x_n,z)\leq \phi(x_n)-\phi(z)+ m_{x_n,z}
\end{equation}
That is $z \in S(x_n)$ for all $n \in \mathbb{N}$ and thus $\alpha(x_n)\leq \phi(z)$. Taking (\ref{BB4}) into account, we obtain $L\leq \phi (z)$. Moreover, by lower semi-continuity of $\phi(x)$ we know that $ \phi (z)\leq L$. Hence,  $ \phi (z)= L$.
Now, by the triangle inequality, the assumption (\ref{BB}) which implies that $Tz \in S(z)$, and that $z \in S(x_n)$ for all $n \in \mathbb{N}$, we obtain

\begin{eqnarray}\label{BB15}
  m(x_n,Tz)-m_{x_n,Tz} &\leq& m(x_n,z)-m_{x_n,z}+m(z,Tz)- m_{z,Tz}\\ \nonumber
    &\leq& \phi(x_n)- \phi(z)+\phi(z) -\phi(Tz)=\phi(x_n) -\phi(Tz). \nonumber
\end{eqnarray}
Hence, $Tz \in S(x_n)$ for all $n \in \mathbb{N}$, which yields that $\alpha(x_n)\leq \phi(Tz)$ for all $n \in \mathbb{N}$
From (\ref{BB4}) we have $\phi(Tz)\geq L$ is obtained and (\ref{BB}) with $x=z$ implies  $\phi(Tz)\leq \phi (z)$. All together with that $\phi(z)=L$, we come to the conclusion that $\phi(Tz)\leq \phi (z)$ and hence from the assumption (\ref{BB}),we have $p(Tz,z)=0$ and hence $Tz=z$.
\end{proof}

\begin{remark}\label{AAstrong}
Note that if in Theorem \ref{AA}, we replace the assumption (\ref{BB}) by
\begin{equation}\label{BBstrong}
  m(x,Tx)\leq  \phi(x)-\phi(Tx),~~~~ \texttt{for all}~ x \in X,
\end{equation}
and assume that $X_0=\{x \in X:m(x,x)=0\}\neq\emptyset$, then the fixed point $z$ must satisfy $m(z,z)=0$.
\end{remark}
The next Theorem corrects the gap in Theorem 6 in \cite{EK02}.
\begin{theorem} \label{dd}
Let $\phi : X\rightarrow [0,\infty )$ be a weak lower semicontinuous on a complete partial metric space $(X,p)$. Then there exists $z\in X$ such that
\begin{center}
$\phi (z)< \phi (x)+p(z,x)-p(z,z)$ for all $x\in X$ with $x\neq z$.
\end{center}
\end{theorem}

\begin{proof}
We consider the element $z\in X$ to be the same one that have been obtained in the proof of Theorem \ref{A}. It is enough to show that $x\not\in S(z)$ for all $x\neq z$. Assume the contrary; that is, there exists $w\neq z$ such that $w\in S(z)$. This implies that
$0<p(z,w)\leq \phi (z)-\phi (w)+p(z,z)$. Thus, $0<p(z,w)-p(z,z)\leq\phi (z)-\phi (w)$ and hence
$$
\phi (w)<\phi (z).
$$

Using triangle inequality, we have
\begin{eqnarray*}
p(x_n,w) &\leq&  p(x_n,z) + p(z,w)- p(z,z) \\
&\leq & \phi (x_n )-\phi (z)+ p(x_n,x_n )+\phi (z)-\phi (w)+p(z,z)-p(z,z) \\
&=& \phi (x_n )-\phi (w)+ p(x_n,x_n )
\end{eqnarray*}
This implies that $w\in S(x_n )$ and hence $\alpha (x_n )\leq \phi(w)$ for all $n\in \mathbb{N}$. Therefore,
$$\lim_{n\rightarrow\infty}\alpha (x_n )=\phi (z)\leq \phi (w),$$
which leads us to a contradiction. \\
Thus, for any $x\in X$, $x\neq z$, we have $x\not\in S(z)$.

\end{proof}
Now we generalize Theorem \ref{dd} to $M$-metric spaces
\begin{theorem} \label{ddm}
Let $\phi : X\rightarrow [0,\infty )$ be a  lower semicontinuous on a complete $M-$ metric space $(X,p)$. Then there exists $z\in X$ such that
\begin{center}
$\phi (z)< \phi (x)+m(z,x)-m_{x,z}$ for all $x\in X$ with $x\neq z$.
\end{center}
\end{theorem}

\begin{proof}
Recalling that $S(z)=\{x \in X: m(x,z)-m_{x,z}\leq \phi(z)-\phi(x)\}$,  It is enough to show that $x \not \in S(z)$ for all $x\neq z$, where $z\in X$ is the same one that have been obtained in the proof of Theorem \ref{AA} with the property that $z \in S(x_n)$ for all $n \in \mathbb{N}$.  Assume the contrary; that is, there exists $w\neq z$ such that $w\in S(z)$. This implies that
$0<m(z,w)-m_{w,z}\leq \phi (z)-\phi (w)$. Thus, $\phi (w)<\phi (z).$

Using triangle inequality, we have for each $n \in \mathbb{N}$
\begin{eqnarray*}
m(x_n,w)-m_{x_n,w} &\leq&  m(x_n,z)-m_{x_n,z} + m(z,w)- m_{z,w} \\
&\leq & \phi (x_n )-\phi (z)+\phi (z)-\phi (w)-\phi(z) \\
&=& \phi (x_n )-\phi (w)
\end{eqnarray*}
This implies that $w\in S(x_n )$ for all $n \in \mathbb{N}$ and hence $\alpha (x_n )\leq \phi(w)$ for all $n\in \mathbb{N}$. Therefore,
$$\lim_{n\rightarrow\infty}\alpha (x_n )=\phi (z)\leq \phi (w),$$
which leads us to a contradiction. \\

\end{proof}

\begin{remark}
Notice that since each partial metric space is $M-$metric space, then the conclusion in Theorem \ref{dd} agrees with the conclusion of Theorem \ref{ddm}. Indeed, if $p$ is a partial metric then its is an $M-$ metric ($p=m$) and from Theorem \ref{dd} we conclude that there exists $z \in X$ such that $\phi(z)< \phi(x)+m(x,z)-m(z,z)\leq \phi(x)+m(x,z)-m_{x,z}$.
\end{remark}
\section{Examples}

In this section, we will give two examples of an $M-$metric space which are not  partial metric spaces that will verify Theorem \ref{AA} and its remark (Remark \ref{AAstrong}).

\begin{example}
Let $ X=\{ 1,2,3,4\}$; define the function on $X\times  X$ as follows
$m(1,1)=1,$ $m(2,2)=3, $ $m(3,3)=5$, $m(4,4)=3,$
$m(1,2)=m(2,1)=10$, $m(1,3)=m(3,1)=m(3,2)=m(2,3)=7$,
$m(1,4)=m(4,1)=8,$  $m(2,4)=m(4,2)= 6$, $m(3,4)=m(4,3)=6$.
Then it is easy to verify that  $m$ is an $M-$metric space but it is not a partial metric space because it does not satisfy the triangle inequality
$m(1,2)\not\leq m(1,3)+m(3,2)-m(3,3)$. \\
In this example we notice that $m(x,x)\neq 0$, for all $x\in X$.
Let $\phi : X\rightarrow \mathbb{R}^+$ defined by $\phi (x)=10x$
and define the self-map $T: X\rightarrow X$ by $T(x)=1$ for all $x\neq 4$ and $T(4)=4$.
Take $\epsilon = 0.1$ and define the corresponding open balls
$B(x,0.1 ) =\{ y\in X : m(x,y)< m_{x,y}+0.1$. It is straightforward to verify that the open balls are single sets;
$B(x,0.1 ) =\{ x\}$ for all $x\in X$. Hence, the $M-$ metric topology on $X$  is the discrete topology and thus each map defined on $X$ is lower semicontinuous.
Also, for all $x\in X$, we notice that $m(x,Tx)\leq m_{x,Tx}+ \phi(x)-\phi(Tx)$.
Therefore, the function $T$ satisfies the conditions of the main result (Theorem \ref{AA}) and so it has a fixed point.
Actually, $x=1, 4$ are fixed points.
\end{example}
The next example is to verify Remark \ref{AAstrong}.
\begin{example}
Let $ X=\{ 1,2,3,4\}$; define the function on $X\times  X$ as follows:
$m(1,1)=0,$ $m(2,2)=3, $ $m(3,3)=5$, $m(4,4)=0,$
$m(1,2)=m(2,1)=10$, $m(1,3)=m(3,1)=m(3,2)=m(2,3)=7$,
$m(1,4)=m(4,1)=8,$  $m(2,4)=m(4,2)= 5$, $m(3,4)=m(4,3)=6$. \\
Note that $X_0=\{x \in X:m(x,x)=0\}\neq \emptyset$.
Then it is easy to verify that  $m$ is an $M$-metric space but it is not a partial metric space because it does not satisfy the triangle inequality
$m(1,2)\not\leq m(1,3)+m(3,2)-m(3,3)$.  \\
Let  $\phi : X\rightarrow \mathbb{R}^+$ defined by $\phi (x)=10x$
and define the self-mapping $T: X\rightarrow X$ by $T(x)=1$ for all $x\neq 4$ and $T(4)=4$.
Moreover, for $\epsilon = 0.1$, it is straightforward to verify that for all $x\in X$, the corresponding open balls are single sets; that is,
$$
B(x,0.1 ) =\{ y\in X : m(x,y)< m_{x,y}+0.1 \} =\{ x\}.
$$

Hence, $X$  has a discrete topology structure.
Also, for all $x\in X$, we notice that $m(x,Tx)\leq \phi(x)-\phi(Tx)$.
Therefore, the function $T$ satisfies the conditions of the main result Theorem \ref{B} and so it has a fixed point.
Actually, $x=1, 4$ are the fixed points.
\end{example}

\end{document}